\documentclass[12pt,reqno]{amsart}
\usepackage{ucs}
\usepackage{amssymb}
\usepackage{amsthm}
\usepackage{amsmath}
\usepackage{latexsym}
\usepackage[cp1251]{inputenc}
\usepackage{graphicx}
\usepackage{wrapfig}
\usepackage{caption}
\usepackage{subcaption}
\usepackage{indentfirst}
\usepackage[left=2.6cm,right=2.6cm,top=3cm,bottom=3cm,bindingoffset=0cm]{geometry}
\usepackage{enumerate}
\usepackage{color}

\def\A{\mathcal{A}}
\def\B{\mathcal{B}}
\def\E{\mathcal{E}}
\def\D{\mathcal{D}}
\def\O{\mathbf{O}}
\def\cS{\mathcal{S}}

\def\Z{\mathbb{Z}}
\def\r{\mathrm{right}}
\DeclareMathOperator{\aut}{Aut}
\DeclareMathOperator{\cay}{Cay}
\DeclareMathOperator{\cyc}{Cyc}
\DeclareMathOperator{\id}{id}
\DeclareMathOperator{\iso}{Iso}

\DeclareMathOperator{\orb}{Orb}
\DeclareMathOperator{\rk}{rk}
\DeclareMathOperator{\Span}{Span}
\DeclareMathOperator{\sym}{Sym}
\DeclareMathOperator{\rad}{rad}
\DeclareMathOperator{\Sup}{Sup}
\DeclareMathOperator{\DCI}{DCI}
\DeclareMathOperator{\CI}{CI}
\DeclareMathOperator{\pr}{pr}
\makeatletter 
\def\@seccntformat#1{\csname the#1\endcsname. } 
\def\@biblabel#1{#1.}

\makeatother

\title[The group $C^4_p\times C_q$ is a $\DCI$-group]
{The group $\boldsymbol{C^4_p\times C_q}$ is a 
$\boldsymbol{\DCI}$-group}

\author{Istv\'an Kov\'acs}
\address{I.\ Kov\'acs 
\newline\indent
UP IAM, University of Primorska, Muzejski trg 2, 6000 Koper, Slovenia 
\newline\indent
UP FAMNIT, University of Primorska, Glagol\v jaska 8, 6000 Koper, Slovenia}
\email{istvan.kovacs@upr.si}
\author{Grigory Ryabov}
\address{G.\ Ryabov
\newline\indent
Sobolev Institute of Mathematics, Acad. Koptyug avenue 4, 630090, Novosibirsk, Russia 
\newline\indent  
Novosibirsk State University, Pirogova 1, 630090, Novosibirsk, Russia
\newline\indent
Mathematical Center in Akademgorodok, Acad. Koptyug avenue 4, 630090, Novosibirsk, Russia}
\email{gric2ryabov@gmail.com}
\thanks{This research work was supported by the Slovenian Research Agency (project no. BI-RU/19-20-032). 
I.\ Kov\'acs was also supported  by 
the Slovenian Research Agency (research program P1-0285 and research 
projects N1-0062, N1-0140, J1-9108 and J1-1695).
G.\ Ryabov was also supported by the Mathematical Center in Akademgorodok.}
\date{}

\newtheorem{lemm}{Lemma}[section]
\newtheorem{theo}[lemm]{Theorem}
\theoremstyle{remark}
\newtheorem{rem}[lemm]{Remark}

\begin{document}
\vspace{\baselineskip}
\vspace{\baselineskip}
\vspace{\baselineskip}
\vspace{\baselineskip}
\begin{abstract}
We prove that the group $C_p^4\times C_q$ is a $\DCI$-group for 
distinct primes $p$ and $q$, that is, two Cayley digraphs over $C_p^4 \times C_q$ are isomorphic if and only if their connection sets are conjugate by a group automorphism.
\\
\\
\textbf{Keywords}: isomorphism, $\DCI$-group, Schur ring.
\\ [+0.08in]
\textbf{MSC}: 05C25, 05C60, 20B25.
\end{abstract}

\maketitle

\section{Introduction}

Let $G$ be a finite group. The \emph{Cayley digraph} $\cay(G,S)$, where $S\subseteq G$, is defined to be the digraph whose vertex set is $G$ and arc set is $\{ (g,sg) : g\in G,\, s\in S\}$.
Two digraphs 
$\cay(G,S)$ and $\cay(G,T)$ are called \emph{Cayley isomorphic} if there exists $\varphi \in \aut(G)$ such that $S^{\varphi}=T$. It is clear that Cayley isomorphic Cayley digraphs are 
also isomorphic, however, the converse implication does not hold in general.  A subset $S \subseteq G$ is called a \emph{$\CI$-subset} if for every $T\subseteq G$, $\cay(G,S) \cong \cay(G,T)$ implies that $\cay(G,S)$ and $\cay(G,T)$ are Cayley isomorphic.  
A finite group $G$ is called a \emph{$\DCI$-group} (\emph{$\CI$-group}, respectively) if all of its subsets (all of its inverse-closed subsets, respectively) are $\CI$-subsets.

These definitions go back to \'Ad\'am~\cite{A}, who conjectured that all finite cyclic groups are $\DCI$-groups. The problem of determining all finite $\CI$-groups was raised by Babai and Frankl~\cite{BF}.
Accumulating the work of several mathematicians, Li et al.~\cite{LLP} 
reduced the candidates of $\CI$-groups to a restricted list.  
We refer to the survey paper \cite{Li} for most results on $\CI$- and $\DCI$-groups.

One of the crucial steps towards the classification of all $\DCI$-groups is to determine the abelian $\DCI$-groups. 
Cyclic $\DCI$-groups were classified by Muzychuk in~\cite{M1,M2}. He proved that the cyclic 
group $C_n$ of order $n$ is a $\DCI$-group if and only if $n \in \{k,2k,4k\}$, where $k$ is an odd square-free number. 
It is known that every Sylow $p$-subgroup of an abelian $\DCI$-group is either elementary abelian or isomorphic to $C_4$ (see \cite{BF,M1,M2}). Since the class of $\DCI$-groups is closed under taking subgroups, it is a natural question which elementary abelian groups are $\DCI$-groups. It was derived in several papers that such groups of rank at most $5$ are $\DCI$-groups. For a detailed account on these results concerning groups of rank up to $4$, we refer to \cite{Li}. 
An elementary proof that $C_p^4$ is a $\DCI$-group was given by Morris~\cite{M}, and the fact that $C_p^5$ is a 
$\DCI$-group was shown recently by Feng and Kov\'acs~\cite{FK}.
On the other hand, the following groups are not $\DCI$-groups: $C_2^6$ (Nowitz~\cite{Now}); $C_3^8$ (Spiga~\cite{Sp}); $C_p^r$, where $r\geq 2p+3$ (Somlai~\cite{Som}). 

Kov\'{a}cs and Muzychuk~\cite{KM} proved that the group $C_p^2\times C_q$ is a $\DCI$-group for every distinct primes $p$ and $q$. In the same paper they conjectured that 
the direct product of two $\DCI$-groups with coprime orders 
is a $\DCI$-group. The conjecture was disproved by Dobson~\cite{Dob1}. On the other hand,  Dobson~\cite{Dob2} also showed that the conjecture is true for abelian groups under strong restrictions on 
the order of the factors.
Recently, Muzychuk and Somlai~\cite{MS} proved that 
$C_p^3\times C_q$ is a $\DCI$-group for every distinct 
primes $p$ and $q$. 
Our goal in this paper is to extend this result and prove the following statement:

\begin{theo}\label{main}
The group $C_p^4\times C_q$ is a $\DCI$-group for distinct primes $p$ and $q$.
\end{theo}

We prove Theorem~\ref{main} using the so called S-ring approach, which was proposed by Klin and P\"{o}schel~\cite{KP}, and developed later by Hirasaka and Muzychuk~\cite{HM}. 
An \emph{S-ring} over a group $G$ is a subring of the group ring 
$\Z G$, which is a free $\Z$-module spanned by a special partition of $G$ (see Section~2 for the exact definition). The notion of an S-ring goes back to Schur~\cite{Schur}. To prove that a given group is a 
$\DCI$-group it is sufficient to show that every member of a certain family of S-rings over that group is a $\CI$-S-ring (see Section~3). 
In fact, we are going to show that the latter family of S-rings 
over $C_p^4 \times C_q$ consists of so called star and   
nontrivial generalized wreath products. Then, we use results on star and generalized wreath products derived in \cite{KM,KR} and also results on S-rings over groups of non-powerful order derived in \cite{MS} 
(a positive integer $n$ is called \emph{powerful} if $p^2$ divides $n$ for every prime divisor $p$ of $n$).
\medskip

The paper is organized as follows: We collect the basic 
definitions and facts from S-ring theory needed in this paper in 
Section~2.  $\CI$-S-rings are the subject of Section~3, where their relation with $\DCI$-groups is also discussed.  Here we prove a new sufficient condition for an S-ring to be a $\CI$-S-ring (see Lemma~\ref{minnorm}). In Section~4, we discuss the generalized wreath and star products of S-rings, and establish a sufficient condition for the generalized wreath product of two S-rings over elementary abelian groups to be cyclotomic (see Lemma~\ref{gwrcycl}). In Sections~5 and 6, we provide information on so called $p$-S-rings over elementary abelian $p$-groups and groups of non-powerful 
order, respectively. Finally, the proof of Theorem~\ref{main} will 
be presented in Section~7.

\subsection*{Notation.}
\medskip

Let $X\subseteq G$. The element $\sum_{x\in X} {x}$ of the group ring $\Z G$ is denoted by $\underline{X}$.

The set $\{x^{-1} : x\in X\}$ is denoted by $X^{-1}$.

The subgroup of $G$ generated by $X$ is denoted by 
$\langle X\rangle$; we also set $\rad(X)=\{g\in G:\ gX=Xg=X\}$.

The order of $g \in G$ is denoted by $o(g)$.

The set $\{ (g,xg): x \in  X,\, g\in G\}$ of arcs of the Cayley digraph 
$\cay(G,X)$ is denoted by $R(X)$.

The group of all permutations of a set $\Omega$ is denoted by 
$\sym(\Omega)$, and the identity element of $\sym(\Omega)$ 
by $\id_\Omega$.

The subgroup of $\sym(G)$ consisting of the right multiplications with 
elements in $G$ is denoted by $G_\r$. 
For a subgroup $H \le G$, the subgroup of 
$G_\r$ consisting of the right multiplications with 
elements in $H$ is denoted by $H_r$.

For a set $\Delta \subseteq \sym(G)$ and a section $S=U/L$ of $G$,  
we set 
$$
\Delta^S=\{f^S:~f\in \Delta,~S^f=S\},
$$
where $S^f=S$ means that $f$ maps $U$ to itself and permutes the 
$L$-cosets in $U$ among themselves (or in other words, the partition 
of $U$ into its $L$-cosets is $f$-invariant), and $f^S$ denotes the bijection of $S$ induced by $f$.

If $K \leq \sym(\Omega)$ and $\alpha \in \Omega$, then the
stabilizer of $\alpha$ in $K$ and the set of all orbits under $K$ on 
$\Omega$ are denoted by $K_{\alpha}$ and $\orb(K,\Omega)$,  respectively.

The cyclic group of order $n$ is denoted by $C_n$.

The class of finite abelian groups all of whose Sylow subgroups are  elementary abelian is denoted by $\E$.

\section{S-rings}

In this section, we provide the reader with background information about S-rings. 
Regarding notation and terminology, we mostly follow~\cite{HM,KR}. 
We refer to the survey paper \cite{MP} 
for more information on S-rings and their link with combinatorics. 

Let $G$ be a finite group and $\Z G$ be the integer group ring. Denote the identity element of $G$ by $e$. A subring  $\A \subseteq \Z G$ is called an \emph{S-ring} (or \emph{Schur ring}) over $G$ if there exists a partition $\cS(\A)$ of $G$ such that:
\smallskip

\begin{enumerate}[{\rm (1)}]
\setlength{\itemsep}{0.4\baselineskip}
\item $\{e\} \in \cS(\A)$, 
\item  if $X \in \cS(\A)$ then $X^{-1} \in \cS(\A)$,
\item $\A=\Span_{\Z}\{\underline{X} :\ X \in \cS(\A)\}$.
\end{enumerate}
\smallskip

\noindent
The elements of $\cS(\A)$ are called the \emph{basic sets} of  
$\A$ and the number $\rk(\A)=|\cS(\A)|$ is called the \emph{rank} of  $\A$. One can easily check that $XY \in \cS(\A)$   
for $X, Y \in \cS(\A)$ with $|X|=1$ or $|Y|=1$.

Let $\A$ be an S-ring over a group $G$. A set $X \subseteq G$ is  called an \emph{$\A$-set} if $\underline{X} \in \A$. A subgroup 
$H \leq G$ is called an \emph{$\A$-subgroup} if $H$ is an 
$\A$-set. With each $\A$-set $X$ one can naturally associate two 
$\A$-subgroups, namely, $\langle X \rangle$ and $\rad(X)$.

Let $L \unlhd U\leq G$. A section $U/L$ is called an \emph{$\A$-section} if $U$ and $L$ are $\A$-subgroups. If $S=U/L$ is an $\A$-section then the module
$$
\A_S=\Span_{\Z}\{ \underline{X}^\pi:~X \in \cS(\A),~
X \subseteq U \},
$$
where $\pi: U \to U/L$ is the canonical epimorphism, is an 
S-ring over $S$.

The S-ring $\A$ over $G$ is called \emph{primitive} if $G$ has no nontrivial proper $\A$-subgroups. Clearly, if $H$ is a maximal 
$\A$-subgroup, then the S-ring $\A_{G/H}$ is primitive. 

Let $\A$ and $\B$ be S-rings over groups $G$ and 
$H$, respectively. A bijection $f : G \rightarrow H$ is called an 
\emph{isomorphism} (also called \emph{combinatorial isomorphism}) 
from $\A$ to $\B$ if
$$
\{R(X)^f : X \in \cS(\A)\}=\{ R(Y) : Y \in \cS(\B) \},
$$
where $R(X)^f=\{( g^f,~h^f) : ~(g,~h)\in R(X)\}$. 
The set of all isomorphisms from $\A$ to $\B$ is denoted by 
$\iso(\A,\B)$. 

The set $\iso(\A,\A)$ is a subgroup of $\sym(G)$, which contains a 
normal subgroup defined as
$$
\aut(\A):=\{ f \in \iso(\A):~R(X)^f=R(X)~\text{for every}~X\in \cS(\A) \}.
$$
This subgroup is called the \emph{automorphism group} of $\A$. 
One can see that $G_\r \leq \aut(\A)$. 
The S-ring $\A$ is called \emph{normal} if 
$G_\r \lhd \aut(\A)$. 

Clearly, if $S$ is an $\A$-section, then $\aut(\A)^S \leq \aut(\A_S)$. 
Denote the group $\aut(\A) \cap \aut(G)$ by $\aut_G(\A)$. 
It easy to check that if $S$ is an $\A$-section, then 
$(\aut_G(\A))^S \leq \aut_S(\A_S)$.

Let $K$ be a subgroup of $\sym(G)$ containing $G_\r$. 
Schur~\cite{Schur} proved that the $\Z$-submodule
$$
V(K, G):=\Span_{\Z}\{\underline{X} :~X\in \orb(K_e,~G)\},
$$
is an S-ring over $G$. An S-ring $\A$ over $G$ is called  \emph{schurian} if $\A=V(K, G)$ for some $K$ such that 
$G_\r \leq K \leq  \sym(G)$. 
If $\A=V(K, G)$ for some $K \leq \sym(G)$ containig $G_\r$ and $S$ is an $\A$-section, then $\A_S=V(K^S, S)$. Therefore, if 
$\A$ is schurian, then so is $\A_S$ for every $\A$-section $S$. 
One can verify that $\A$ is schurian if and only if 
$\A=V(\aut(\A), G)$.

An S-ring $\A$ over a group $G$ is called \emph{cyclotomic} if there exists $M \leq\aut(G)$ such that $\cS(\A)=\orb(M,\, G)$. In this case 
$\A$ is denoted by $\cyc(M,\, G)$. Obviously, 
$\A=V(G_\r M, G)$. So every cyclotomic S-ring is schurian. If $\A=\cyc(K,\, G)$ for some 
$K \leq \aut(G)$ and $S$ is an $\A$-section, then $\A_S=\cyc(K^S,\, G)$. Therefore, if $\A$ is cyclotomic, then so is $\A_S$ for every $\A$-section $S$. 

Let $K_1, K_2 \leq \sym(\Omega)$ be arbitrary permutation groups of a 
set $\Omega$. We say that $K_1$ and $K_2$ are \emph{$2$-equivalent}, denoted by $K_1 \approx_2 K_2$, if 
$\orb(K_1, \Omega^2)=\orb(K_2, \Omega^2)$, where both 
$K_1$ and $K_2$ act on $\Omega^2$ with the usual  componentwise action. This is an equivalence relation on the set of all subgroups of 
$\sym(\Omega)$, and every equivalence class has a unique maximal element. The unique element of the class containing a given $K \le \sym(\Omega)$ is called the \emph{$2$-closure} of $K$, denoted by $K^{(2)}$. 
If $\A=V(K, G)$ for some $K \leq \sym(G)$ containing $G_\r$, then $K^{(2)}=\aut(\A)$. 
An S-ring $\A$ over $G$ is called 
\emph{ $2$-minimal} if
$$
\big\{ K \leq \sym(G) :~K \geq G_\r~\text{and}~K \approx_2 \aut(\A) \big\}=\{ \aut(\A) \}.
$$

Two groups $K_1, K_2 \leq \aut(G)$ are defined to be \emph{Cayley equivalent}, denoted by $K_1 \approx_{\rm Cay} K_2$, if 
$\orb(K_1, G)=\orb(K_2, G)$. If $\A=\cyc(K, G)$ for some $K \leq \aut(G)$, then $\aut_G(\A)$ is the largest group which is 
Cayley equivalent to $K$. A cyclcotomic S-ring $\A$ over $G$ is 
called \emph{Cayley minimal} if
$$
\big\{ K \leq \aut(G) :~K \approx_{Cay} \aut_G(\A) \big\}=\{ \aut_G(\A) \}.
$$

\section{$\CI$-S-rings}

Let $G$ be a finite group and $X \subseteq G$ be any subset. 
Babai~\cite{Babai} gave the following criterion for $X$ to be a $\CI$-subset. 

\begin{lemm}\label{B}
{\rm (cf.\ \cite[Lemma~3.1]{Babai})}
A subset $X \subset G$ is a $\CI$-subset if and only if 
any two regular subgroups of $\aut(\cay(G, X))$, which are isomorphic to $G$, are conjugate in $\aut(\cay(G, X))$. 
\end{lemm}

In what follows, we write $\Sup(G_\r)=\{K \leq \sym(G):~ K \geq G_\r\}$, and let $\Sup_2(G_\r)$ denote its subset consisting of the $2$-closed groups, or formally,  
$$
\Sup_2(G_\r)=\{K \in \Sup(G_\r):~ K^{(2)}=K\}.
$$

A group $K \in \Sup(G_\r)$ is said to be 
\emph{$G$-transjugate} if for every $\gamma \in \sym(G)$, the inclusion 
$(G_\r)^\gamma \leq K$ implies that $(G_\r)^\gamma$ and 
$G_\r$ are conjugate in $K$ (see \cite[Definition~1]{HM}). 

Let $K_1, K_2\in \Sup(G_\r)$ such that $K_1 \le K_2$. Then 
$K_1$ is said to be a \emph{$G$-complete  
subgroup} of $K_2$, denoted by $K_1 \preceq_G K_2$, 
if for every $\gamma \in \sym(G)$, the inclusion 
$(G_\r)^\gamma \leq K_2$ implies $(G_\r)^{\gamma \delta} \leq K_1$ for a suitable $\delta \in K_2$ (see \cite[Definition~2]{HM}).  The relation 
$\preceq_G$ is a partial order on $\Sup(G_\r)$. The set of the 
minimal elements of the poset $(\Sup_2(G_\r),\preceq_G)$ is denoted by 
$\Sup_2^{\min}(G_\r)$. 

Let $\A$ be an S-ring over $G$. 
The set of all isomorphisms from $\A$ onto some S-ring over $G$ will 
be denoted by $\iso(\A)$, or formally, 
$$
\iso(\A)=\{ f \in \sym(G):~ f \in \iso(\A,\A'),~\A'~\text{is an S-ring over}~G\}.
$$
Let $\iso_e(\A)=\{f \in \iso(\A) : e^f=e\}$. 

Note that, for any $f \in \iso_e(\A)$ and any two basic sets 
$X, Y \in \cS(\A)$, $(XY)^f=X^f Y^f$ (see \cite[page 347]{HM}).
Furthermore, if $S$ is an $\A$-section and 
$f \in \iso_e(\A)$ such that $S^f=S$, then 
$f^S \in \iso_{\bar{e}}(\A_S)$, where $\bar{e}$ denotes the identity 
element of $S$ (see \cite[Proposition~2.7]{HM}).

One can see that $\aut(\A)\aut(G) \subseteq \iso(\A)$. 
The S-ring $\A$ is defined to be a \emph{$\CI$-S-ring} if 
$\iso(\A)=\aut(\A) \aut(G)$ (see \cite[Definition~3]{HM}). 
For instance, it is easy to check that 
$\Z G$ and the S-ring of rank~$2$ over $G$ are $\CI$-S-rings. 

\begin{lemm}\label{trans}
{\rm (cf.\ \cite[Theorem~2.6]{HM})}
Let $K \in \Sup_2(G_\r)$ and $\A=V(K, G)$. 
The following conditions are equivalent:
\begin{enumerate}[{\rm (1)}]
\item $K$ is $G$-transjugate. 
\item $\iso(\A)=\aut(\A) \aut(G)$.
\item $\iso_e(\A)=\aut(\A)_e \aut(G)$.
\end{enumerate}
\end{lemm}

It is well-known that $\aut(\cay(G,X)) \in \Sup_2(G_\r)$ 
for every $X \subseteq G$. Therefore, combining together    
Lemma~\ref{B}, the definition of $\preceq_G$ and Lemma~\ref{trans} yields the following implication:

\begin{lemm}\label{cimin}
If $V(K, G)$ is a $\CI$-S-ring for every 
$K \in \Sup_2^{\rm min}(G_\r)$, then $G$ is a $\DCI$-group.
\end{lemm}

In fact, we are going to derive Theorem~\ref{main} by 
showing that the condition in the lemma above holds when 
$G \cong C_p^4 \times C_q$, where $p$ and $q$ are distinct primes.

\begin{rem}\label{cisrings}
The condition that $V(K, G)$ is a $\CI$-S-ring for every 
$K \in \Sup_2^{\rm min}(G_\r)$ is equivalent to say that 
every schurian S-ring over $G$ is $\CI$. 
In proving Theorem~\ref{main} we will use the fact that 
the schurian S-rings over each proper subgroup of $C_p^4 \times C_q$ are 
$\CI$. Here are the references:
\begin{enumerate}[]
\setlength{\itemsep}{0.4\baselineskip}
\item $C_p^n$, $p >2$ and $n \le 4$ 
(Hirasaka and Muzychuk~\cite[Theorem~1.3]{HM});
\item $C_2^n$ and $n \le 4$ (this follows from~\cite[Propostion~1.4]{HM} and Lemma~\ref{2srings});
\item $C_p \times C_q$ (Klin and P\"{o}schel~\cite{KP});
\item $C_p^2 \times C_q$ 
(Kov\'acs and Muzychuk~\cite[Theorem~2.2]{KM});
\item $C_p^3 \times C_q$ 
(Muzychuk and Somlai~\cite[Theorem~16]{MS}).
\end{enumerate}
\end{rem}

The following lemma is an easy consequence of the 
definitions.

\begin{lemm}\label{cinormal}
Let $\A$ be a schurian S-ring over an abelian group $G$ and 
$S$ be an $\A$-section such that at least one of the following conditions  
holds:
\begin{enumerate}[{\rm (1)}]
\item $\A$ is a $\CI$-S-ring and $S_\r \lhd \aut(\A)^S$.
\item $\A_S$ is a normal $\CI$-S-ring.
\end{enumerate} 
Then, for every $f \in \iso_e(\A)$ such that $S^f=S$, 
$f^S \in \aut(S)$. In particular, $f^S=\id_S$ whenever 
$f^S$ fixes pointwise a generating set of $S$.
\end{lemm}

\begin{proof}
Then $f^S \in \iso_{\bar{e}}(\A_S)$, where $\bar{e}$ denotes the identity element of $S$. 

Assume first that condition (1) holds. Then 
$f=\gamma \varphi$ for some $\gamma  \in \aut(\A)_e$ and 
$\varphi \in \aut(G)$ because $\A$ is a $\CI$-S-ring.
Since $e^f=e$, it follows that $e^\gamma=e$, hence $S^\gamma = S$, 
and thus $S^\varphi=S$. Then $f^S = \gamma^S \varphi^S$, where 
$\gamma^S  \in \aut(\A)^S$ and $\varphi^S \in \aut(S)$. 
Since $S_\r \lhd \aut(\A)^S$, $\gamma^S \in \aut(S)$, and 
hence $f^S \in \aut(S)$ as well. 

Now, suppose that condition (2) holds. 
Then $f^S \in \aut(\A_S)_{\bar{e}} \aut(S)=\aut(S)$, where the inclusion holds because $\A_S$ is a $\CI$-S-ring, whereas the equality holds because $\A_S$ is also normal. 
\end{proof}

We finish this section with another lemma about $\CI$-S-rings. 
In order to formulate the lemma, it is necessary to introduce some further notation. 
Let $\A$ be a schurian S-ring over an abelian group $G$, and $L$ be 
an $\A$-subgroup of $G$. Then, the partition of $G$ into the  
$L$-cosets is $\aut(\A)$-invariant. The \emph{kernel} of the action of 
$\aut(\A)$ on the latter cosets is denoted by $\aut(\A)_{G/L}$, that is, 
\begin{equation}\label{eq:kernel}
\aut(\A)_{G/L}=\{\gamma \in \aut(\A) :~(Lg)^\gamma=Lg~\text{for every}~g \in G\}. 
\end{equation}
Since $\aut(\A)_{G/L} \lhd \aut(\A)$, the group $K:=\aut(\A)_{G/L}G_{\rm right}$ can be formed. Clearly, $K \leq \aut(\A)$. It follows from~\cite[Proposition~2.1]{HM} that $K=K^{(2)}$.

\begin{lemm}\label{minnorm}
Let $\A$ be a schurian S-ring over an abelian group $G$, 
$L$ be an $\A$-subgroup of $G$, and let 
$K=\aut(\A)_{G/L}G_{\rm right}$.
Suppose that both $\A_{G/L}$ and $V(K, G)$ are $\CI$-S-rings and 
$\A_{G/L}$ is $2$-minimal.  Then $\A$ is a $\CI$-S-ring.
\end{lemm}

\begin{proof}
We first prove that $\aut(\A)^{G/L}$ is $G/L$-transjugate. 
The groups $\aut(\A_{G/L})$ and $\aut(\A)^{G/L}$ are $2$-equivalent because 
$\A_{G/L}=V(\aut(\A)^{G/L},L)$. Then, by the $2$-minimality of $\A_{G/L}$, $\aut(\A)^{G/L}=\aut(\A_{G/L})$. Since $\A_{G/L}$ is a $\CI$-S-ring, 
Lemma~\ref{trans} yields that $\aut(\A)^{G/L}=\aut(\A_{G/L})$ is 
$G/L$-transjugate.

Now, let us show that $K \preceq_G \aut(\A)$. 
Let $H$ be a regular subgroup of $\aut(\mathcal{A})$ such 
that $H \cong G$. Then $H^{G/L}$ is an abelian and transitive subgroup of 
$\aut(\A)^{G/L}$, and hence it is regular on $G/L$. 
Therefore, $H^{G/L} \cong (G/L)_\r=(G_\r)^{G/L}$. 
There exists $\gamma \in \aut(\A)$ such that $(H^{G/L})^{\gamma^{G/L}}=(G/L)_\r=(G_\r)^{G/L}$ because $\aut(\A)^{G/L}$ is $G/L$-transjugate. 
This implies that $H^\gamma \leq K$, and thus 
$K \preceq_G \aut(\A)$.

Finally, we prove that $\aut(\A)$ is $G$-transjugate. 
This is equivalent to say that $\A$ is a $\CI$-S-ring by Lemma~\ref{trans}. 
Again, let $H$ be a regular subgroup of $\aut(\A)$ such that $H \cong G$. Since 
$K \preceq_G \aut(\A)$, there exists $\gamma \in\aut(\A)$ such that 
$H^\gamma \leq K$. The S-ring $V(K,\, G)$ is a $\CI$-S-ring by one of the  assumptions of the lemma. So $K=\aut(V(K,\, G))$ is $G$-transjugate by Lemma~\ref{trans}. Therefore, $H^\gamma$ and $G_\r$ are conjugate in $K$, hence in 
$\aut(\A)$ as well. The lemma is proved.
\end{proof}

\section{Wreath and star products}

Let $\A$ be an S-ring over a group $G$ and $S=U/L$ be an 
$\A$-section of $G$. The S-ring~$\A$ is called the \emph{$S$-wreath product} or the \emph{generalized wreath product} of $\A_U$ and 
$\A_{G/L}$ if $L \lhd G$ and $L \leq \rad(X)$ for each basic set 
$X$ outside~$U$. In this case we write $\A=\A_U \wr_{S} \A_{G/L}$, and omit $S$ when $U=L$. The $S$-wreath product is called \emph{nontrivial} (or \emph{proper})  if $L \neq \{e\}$ and $U \neq G$.  We say that an S-ring $\A$ is \emph{decomposable} if it is the nontrivial $S$-wreath product for some $\A$-section $S$ of $G$. The construction of the generalized wreath product of S-rings was introduced  in~\cite{EP} and independently in~\cite{LM} under the name \emph{wedge product}.  

\begin{lemm}\label{cigwr}
{\rm (\cite[Theorem~1.1]{KR})}
Let $G \in \E$, $\A$ be an S-ring over $G$, and $S=U/L$ be an 
$\A$-section of $G$. Suppose that $\A$ is the nontrivial $S$-wreath product and the S-rings $\A_U$ and $\A_{G/L}$ are $\CI$-S-rings. 
Then $\A$ is a $\CI$-S-ring whenever 
$$
\aut_S(\A_S)=\aut_U(\A_U)^{S} \aut_{G/L}(\A_{G/L})^S.
$$ 
In particular,  $\A$ is a $\CI$-S-ring if $\aut_S(\A_S)=\aut_U(\A_U)^S$ or $\aut_S(\A_S)=\aut_{G/L}(\A_{G/L})^S$.
\end{lemm}

Note that, if $U=L$ in Lemma~\ref{cigwr}, then $\A$ is a 
$\CI$-S-ring because the group $\aut_S(\A_S)$ is trivial.

\begin{lemm}\label{cicayleymin}
With the notations in Lemma~\ref{cigwr}, suppose that at least one of the S-rings 
$\A_U$ and $\A_{G/L}$ is cyclotomic and $\A_S$ is Cayley minimal. 
Then $\A$ is a $\CI$-S-ring.
\end{lemm}

\begin{proof}
Suppose that $\A_{G/L}$ is cyclotomic. Then $\A_S$ is also cyclotomic. 
Clearly, $\A_S=\cyc(\aut_{G/L}(\A_{G/L})^S,\, S)$. Therefore, 
$$
\aut_{G/L}(\A_{G/L})^S \approx_{\cay} \aut_S(\A_S).
$$ 
Since $\A_S$ is Cayley minimal, $\aut_{G/L}(\A_{G/L})^S=\aut_S(\A_S)$ and we are done by Lemma~\ref{cigwr}. 
Similarly, in the case when $\A_U$ is cyclotomic, we obtain that 
$\aut_U(\A_U)^S=\aut_S(\A_S)$ and $\A$ is a $\CI$-S-ring by 
Lemma~\ref{cigwr}. The lemma is proved.
\end{proof}

\begin{lemm}\label{gwrcycl}
Let $G$ be an elementary abelian $p$-group, $\A$ be an S-ring over 
$G$, and $S=U/L$ be an $\A$-section of $G$. Suppose that 
$\A$ is the nontrivial $S$-wreath product, $\A_U=\cyc(K_1,\, U)$ for some $K_1\leq \aut(U)$, and $\A_{G/L}=\cyc(K_2,\, G/L)$ for some $K_2 \leq \aut(G/L)$. Then $\A$ is cyclotomic whenever $K_1^S=K_2^S$.
\end{lemm}

\begin{proof}
Since $G$ is elementary abelian, there exist groups $D$ and $V$ such that $G=D \times U$ and $U=V \times L$. Let $D=\langle x_1 \rangle \times \ldots \times \langle x_t \rangle$, where $o(x_i)=p$ for every 
$i \in \{1,\ldots,t\}$. Choose $(\varphi,\psi) \in K_1 \times K_2$ with 
$\varphi^S=\psi^S$. Suppose that $(x_iL)^{\psi}=y_iz_iL$, where 
$y_i\in D$, $z_i\in V$, and $i \in \{1,\ldots,t\}$. The elements $y_iz_i$, $i\in\{1,\ldots,t\}$, generate a group $D^{'}$ of rank~$t$ because $\rk(D)=t$ and $\psi \in \aut(G/L)$. Note that, $D^{'}\cap U=\{e\}$. Indeed, let $g \in D^{'} \cap U$. Then $(gL)^{\psi^{-1}} \in U/L \cap D/L=\{L\}$ because $(U/L)^{\psi}=U/L$. So $g \in L$. However, 
$D^{'} \leq D\times V$ and hence $g \in  (D\times V) \cap L=\{e\}$. This means that $G=D^{'} \times U$. The group $G$ is elementary abelian and hence there exists $\alpha(\varphi,\psi) \in \aut(G)$ such that
$$
\alpha(\varphi,\psi)^U=\varphi,~(x_i)^{\alpha(\varphi,\psi)}=
y_iz_i,~i \in \{1,\ldots,t\}.
$$

For each $i \in \{1,\ldots,t\}$ and each $l \in L$, define $\beta(i,l) \in \aut(G)$ in the following way:
$$
\beta(i,l)^U=\id_{U},~x_j^{\beta(i,l)}=x_jl^{\delta_{ij}},
$$
where $\delta_{ij}$ is the Kronecker delta.

Put
$$
K=\big\langle \alpha(\varphi,\psi),\beta(i,l):~(\varphi,\psi) \in K_1 \times K_2,~\varphi^S=\psi^S,~i \in \{1,\ldots,t\},~l \in L \big\rangle \leq \aut(G).
$$
We finish the proof by showing that $\A=\cyc(K,G)$. 

We prove first that every basic set of $\A$ is $K$-invariant. Let 
$X \in \cS(\A)$. Let $\alpha=\alpha(\varphi,\psi)$, where 
$(\varphi,\psi) \in K_1\times K_2$ and $\varphi^S=\psi^S$. The definition of $\alpha$ implies that $L^\alpha=L$, $U^\alpha=U$, and 
for every $d\in D$, 
\begin{equation}\label{eq1}
(dL)^{\psi}=d^{\alpha}L.
\end{equation}

If $X \subseteq U$ then $X^\alpha=X^\varphi=X$. The last equality holds because $\A_U=\cyc(K_1,\, U)$ and $\varphi \in K_1$. 
Suppose that $X$ lies outside $U$.  Then $L \leq \rad(X)$. So  $X=d_1v_1L \cup \ldots \cup d_sv_sL$ for some $d_i \in D$ and 
$v_i \in V$, $i \in \{1,\ldots,s\}$. Note that 
\begin{equation}\label{eq2}
X^\alpha=d_1^{\alpha}v_1^{\alpha}L\cup \ldots \cup d_s^{\alpha}v_s^{\alpha}L=d_1^{\alpha}v_1^{\varphi}L\cup \ldots~\cup d_s^{\alpha}v_s^{\varphi}L.
\end{equation}
Due to $\varphi^{U/L}=\psi^{U/L}$, we have for every $i \in \{1,\ldots,s\}$,
\begin{equation}\label{eq3}
v_i^{\varphi}L=(v_iL)^{\varphi}=(v_iL)^{\psi}.
\end{equation} 
Denote the canonical epimorphism from $G$ to $G/L$ by $\pi$. Using Eqs.~\eqref{eq1}--\eqref{eq3}, a  straightforward check shows that
\begin{eqnarray}
\nonumber \pi(X^{\alpha})=\{d_1^{\alpha}v_1^{\varphi}L,\ldots,d_s^{\alpha}v_s^{\varphi}L\}=\{(d_1L)^{\psi}(v_1L)^{\psi},\ldots,(d_sL)^{\psi}(v_sL)^{\psi}\}=\\
\nonumber \{d_1v_1L,\ldots,d_sv_sL\}^{\psi}=\{d_1v_1L,\ldots,d_sv_sL\}=\pi(X).
\end{eqnarray}
The fourth equality holds because $\A_{G/L}=\cyc(K_2,\, G/L)$ and 
$\psi \in K_2$. Since $L \leq \rad(X)$ and $L \leq \rad(X^\alpha)$, 
we obtain that $X^\alpha=X$.

Let $\beta=\beta(i,l)$ for some $i \in \{1,\ldots,t\}$ and $l \in L$. If 
$X \subseteq U$, then $X^\beta=X$ because $\beta$ acts trivially on 
$U$. If $X$ lies outside $U$, then $X^\beta=X$ because 
$L \leq \rad(X)$. Therefore, every basic set of $\A$ is invariant under the action of the generators of $K$ and hence under the action of 
$K$ as well, as claimed.

Now, we prove that every basic set of $\A$ is equal to an orbit under 
$K$. For this purpose it is sufficient to show that $K$ acts transitively on every  basic set. Again, let $X \in \cS(\A)$. If $X \subseteq U$, then $K$ is transitive on $X$ because $K^U=K_1$ and $X$ is an orbit under $K_1$. 

Suppose that $X$ lies outside $U$. Then $L \leq \rad(X)$ and  $X=d_1v_1L\cup \ldots \cup d_sv_sL$, where $d_i \in D$ and 
$v_i \in V$. Given an arbitary $x \in X$, we have to find $\sigma \in K$ such that
 $(d_1v_1)^\sigma=x$. 
 
Let $x=d_iv_i l$, where $i \in \{1,\ldots,s\}$ and $l \in L$. The set $\{d_1v_1L,\ldots,d_sv_sL\}$ is a basic set of 
$\A_{G/L}$, and hence it is an orbit under $K_2$. This yields that there exists $\psi \in K_2$ with $(d_1v_1L)^{\psi}=d_iv_iL$. Since $K_1^S=K_2^S$, there exists $\varphi \in K_1$ with 
$\varphi^S=\psi^S$. Put $\alpha=\alpha(\varphi,\psi)$. Using Eqs.~\eqref{eq1} and~\eqref{eq3}, we obtain that $(d_1v_1)^\alpha L=(d_1v_1L)^\psi=d_iv_i L$. So $(d_1v_1)^\alpha=d_iv_il^{'}$ for some $l^{'} \in L$. 

Then $d_i=x_1^{r_1} \cdots x_t^{r_t}$ for some numbers 
$r_1,\ldots,r_t \in \{0,\ldots,p-1\}$. At least one of the numbers 
$r_i$'s, say $r_m$, is nonzero because $d_iv_i \notin U$. Since $L$ is  an elementary abelian $p$-group and $r_m$ is coprime to $p$, there exists $l^{''} \in L$ such that $(l^{''})^{r_m}=l(l^{'})^{-1}$. 

Put $\beta=\beta(m,l^{''})$ and $\sigma=\alpha\beta$. Then 
$$
(d_1v_1)^{\sigma}=(d_1v_1)^{\alpha\beta}=(d_iv_il^{'})^{\beta}=d_iv_i(l^{''})^{r_m}l^{'}=d_iv_il=x.
$$
Therefore, $K$ indeed acts transitively on $X$. Thus, every basic set of $\A$ is an orbit under $K$, and so $\A=\cyc(K,\, G)$. The lemma is proved. 
\end{proof}

Let $V$ and $W$ be $\A$-subgroups. We say that $\A$ is the \emph{star product} of $\A_V$ and $\A_W$ and write 
$\A=\A_V \star \A_W$ if the following conditions hold:
\smallskip

\begin{enumerate}[{\rm (1)}]
\setlength{\itemsep}{0.4\baselineskip}
\item $V\cap W\lhd W$;
\item each $X \in \cS(\A)$ with $X \subseteq (W\setminus V) $ is a union of some $(V\cap W)$-cosets;
\item for each $X \in \cS(\A)$ with $X \subseteq G\setminus (V\cup W)$ there exist $Y \in \cS(\A_V)$ and $Z \in \cS(\A_W)$ such that 
$X=YZ$.
\end{enumerate}
\smallskip

The star product is called \emph{nontrivial} if $V\neq \{e\}$ and 
$V \neq G$.  Notice that, if $V \cap W \neq \{e\}$ and $V \ne G$, 
then $\A$ is the nontrivial $V/(V\cap W)$-wreath product. The construction of the star product of S-rings was introduced  in~\cite{HM}.

\begin{lemm}\label{cistar}
{\rm (cf. \cite[Proposition~3.2~and~Theorem~4.1]{KM})}
Let $G \in \E$ and $\A$ be a schurian S-ring over $G$. Suppose that 
$\A=\A_V \star \A_W$ for some $\A$-subgroups $V$ and $W$ of 
$G$, and that the S-rings $\A_V$ and $\A_{W/(V\cap W)}$ are 
$\CI$-S-rings. Then $\A$ is a $\CI$-S-ring.
\end{lemm} 

In the special case when $V \cap W=\{e\}$ the star product 
$\A_V \star \A_W$ is the usual \emph{tensor product}, denoted by 
$\A_V \otimes \A_W$ (see~\cite[p.\ 5]{KR}).

\begin{lemm}\label{tenspr}
{\rm (\cite[Lemma 2.8]{FK})}
Let $\A$ be an S-ring over an abelian group $G=G_1\times G_2$, and 
suppose that $G_1$ and $G_2$ are $\A$-subgroups. Then 
$\A=\A_{G_1} \otimes \A_{G_2}$ whenever $\A_{G_1}$ or 
$\A_{G_2}$ is the group ring.
\end{lemm}

\section{$p$-S-rings}

Let $p$ be a prime number. An S-ring $\A$ over a $p$-group $G$ is defined to be a \emph{$p$-S-ring} if every basic set of $\A$ has $p$-power size.

\begin{lemm}\label{minpring}
{\rm (\cite[Lemma~5.2]{KM})}
Let $G$ be an abelian group, $K \in \Sup_2^{\min}(G_\r)$ and 
$\A=V(K, G)$. Suppose that $H$ is an 
$\A$-subgroup of $G$ such that $G/H$ is a $p$-group for some prime $p$. Then $\A_{G/H}$ is a $p$-S-ring.
\end{lemm}

\renewcommand{\arraystretch}{1.5}
\begin{table}
\begin{tabular}{|c|c|c|c|}
\hline
{\rm no.} & $\A$ & {\rm decomposable} & $|\O_\theta(\A)|$    
\\ \hline 
{\rm 1.} & $\Z C_p^3$ & {\rm no} & $p^3$ 
\\ \hline
{\rm 2}. & $\Z C_p^2 \wr \Z C_p$ & {\rm yes} & $p^2$ 
\\  \hline
{\rm 3.} & $\Z C_p \wr \Z C_p^2$  & {\rm yes} & $p$ 
\\  \hline
{\rm 4.} & $(\Z C_p \wr \Z C_p) \otimes \Z C_p$  & yes & $p^2$ 
\\  \hline
{\rm 5.} & $\Z C_p \wr \Z C_p \wr \Z C_p$  & yes  & $p$
\\  \hline
{\rm 6.} & $\cyc(\langle \sigma \rangle,\, C_p^3),~\sigma=$ 
{\footnotesize $\left(\begin{smallmatrix} 1& 1& 0\\ 0& 1& 1\\ 0& 0&1 
\end{smallmatrix}\right)$} & {\rm no}  & $p$ 
\\ \hline
\end{tabular}
\\ [+1ex]
\caption{$p$-S-rings over $C_p^3$.}
\end{table}

The \emph{thin radical} $\O_\theta(\A)$ of an S-ring $\A$ over a group 
$G$ is the $\A$-subgroup defined as
$$
\O_\theta(\A)=\{ x \in G:~\{x\} \in \cS(\A) \}.
$$  

Until the end of this section $G$ stands for an elementary abelian 
$p$-group of rank~$n$ and $\A$ stands for a $p$-S-ring over $G$.
Given a generating set $\{a_1,\ldots,a_n\}$ of $G$, 
the $n$-tuple $(a_1,\ldots,a_n)$ is called  
an \emph{$\A$-basis} if for each $i \in \{1,\ldots,n\}$,
the subgroup $\langle a_1, \ldots, a_i \rangle$ is an 
$\A$-subgroup.

\begin{lemm}\label{p3}
With the above notations, let $n\leq 3$. 
Then $\A$ is cyclotomic and 
\begin{enumerate}[{\rm (1)}]
\setlength{\itemsep}{0.4\baselineskip}
\item if $n=1$ then $\A \cong \Z C_p$;
\item if $n=2$ then $\A \cong \Z C_p^2$ or 
$\A \cong \Z C_p \wr \Z C_p$;
\item if $n=3$ and $p$ is odd, then $\A$ is isomorphic to one of the 
S-rings listed in Table~$1$; if $n=3$ and $p=2$, then $\A$ is 
is isomorphic to one of the S-rings in rows nos.~1--5 in Table~1.
\end{enumerate}
\end{lemm}

\begin{proof}
If $p$ is odd, then 
Hirasaka and Muzcyhuk~\cite{HM} classified the schurian S-rings 
over $C_p^n$, $n \le 3$, and it was proved later by Spiga and Wang~\cite{SW} that all 
$p$-S-rings over $C_p^3$ are schurian. 
If $p=2$, then the statement of the lemma can be verified with the 
help of the GAP package COCO2P~\cite{GAP}.
\end{proof}

For later use we set $\D_p=\cyc(\langle \sigma \rangle,\, C_p^3)$, 
the S-ring in row no.~6 in Table~1. 

The following two lemmas for an odd prime $p$ are \cite[Lemma~5.5]{KR} 
and \cite[Lemma~5.6]{KR}, respectively. Their proofs are also valid for $p=2$. 

\begin{lemm}\label{p3caymin}
{\rm (cf.\ \cite[Lemma~5.5]{KR})}
With the above notations, let  $n \leq 3$. Then $\A$ is Cayley minimal except for the case when $n=3$ and 
$\A \cong \Z C_p \wr \Z C_ p\wr \Z C_p$ (the S-ring no.~5 in Table~1).
\end{lemm}

\begin{lemm}\label{p3notcaymin}
{\rm (cf.\ \cite[Lemma~5.6]{KR})}
With the above notations, let $n=3$ and 
$\A \cong \Z C_p \wr \Z C_p \wr \Z C_p$. Then $|\aut_G(\A)|=p^3$.
\end{lemm}

\begin{lemm}\label{p4minimal}
With the above notations, let $n=4$ 
and suppose that $\A$ is indecomposable and schurian. Then $\A$ is $2$-minimal.
\end{lemm}

\begin{proof}
The statement of the lemma for an odd prime $p$ follows from~\cite[Theorem~4.1]{FK}. If $p=2$, then using the GAP package COCO2P~\cite{GAP}, one can find all indecomposable $p$-S-rings over $G$ (there are 
four of them up to isomorphism), and check that their automorphism groups 
do not contain a $2$-equivalent proper subgroup. So the statement of the lemma also holds for $p=2$. 
\end{proof}

\begin{lemm}\label{p4decomp}
With the above notations, let $n=4$ and suppose that $\A$ is decomposable. Then $\A$ is cyclotomic.
\end{lemm}

\begin{proof}
Let $\A$ be the nontrivial $S$-wreath product for some 
$\A$-section $S=U/L$. Clearly, $|U|,|L| \in \{p,p^2,p^3\}$ and 
$|S| \in \{1,p,p^2\}$. Lemma~\ref{p3} implies that the S-rings 
$\A_U$, $\A_{G/L}$ and $\A_S$ are cyclotomic. 
So $\A_U=\cyc(\aut_U(\A_U),\, U)$, $\A_{G/L}=\cyc(\aut_{G/L}(\A_{G/L}),\, G/L)$, and $\A_S=\cyc(\aut_S(\A_S),\, S)$. According to Lemma~\ref{p3caymin}, the S-ring $\A_S$ is Cayley minimal. 
Therefore, $\aut_U(\A_U)^S=\aut_S(\A_S)=\aut_{G/L}(\A_{G/L})^S$. Now, Lemma~\ref{gwrcycl} yields that $\A$ is cyclotomic. 
The lemma is proved.
\end{proof}

\begin{lemm}\label{2srings}
With the above notations, let $p=2$ and $n \leq 4$. Then $\A$ is a 
$\CI$-S-ring.
\end{lemm}

\begin{proof}
The S-ring $\A$ is schurian by~\cite[Theorem~1.2]{EKP}. 

Suppose first  that $\A$ is decomposable, that is, $\A$ is the nontrivial $S$-wreath product for some $\A$-section $S$ with $|S|\leq p^2$. 
Then $\A$ is cyclotomic by Lemma~\ref{p3} if $n\leq 3$ and by Lemma~\ref{p4decomp} if $n=4$. The S-ring $\A_S$ is Cayley minimal by Lemma~\ref{p3caymin}. So $\A$ is a $\CI$-S-ring by Lemma~\ref{cicayleymin}.

Now, suppose that $\A$ is indecomposable. If $n\leq 3$ then $\A=\mathbb{Z}G$ by Lemma~\ref{p3}, which is obviously a $\CI$-S-ring. Let $n=4$. 
A computation with COCO2P~\cite{GAP} shows that, up to isomorphism, there  
are three indecomposable S-rings over $G$ distinct from $\Z G$. 
One of these S-rings, say $\A_1$, satisfies the conditions $|\aut(\A_1)|=32$ and $G_\r$ is the unique regular subgroup of $\aut(\A_1)$, which is isomorphic to $G$. Since $\A_1$ is schurian, we conclude that $\A_1$ is a $\CI$-S-ring by Lemma~\ref{trans}. The $\CI$-property of the other two indecomposable S-rings follows from the $\CI$-property of $\A_1$ and Lemma~\ref{minnorm}. 
The lemma is proved.
\end{proof}

We conclude this section with a particular S-ring over $C_p^4$, which will appear in our proof of Theorem~\ref{main} in Section~7.

\begin{lemm}\label{p3exceptcase}
With the above notations, let $n=4$ and suppose that 
$\A$ is decomposable, $\A_U \cong \Z C_p \wr \Z C_p \wr \Z C_p$ 
for an $\A$-subgroup $U$, and $\A$ has a basic set outside $U$ 
with trivial radical. 
Then $\aut_G(\A)^U=\aut_U(\A_U)$, unless $p > 2$ and 
there exist $\A$-subgroups $U_1$ and $U_2$ of $G$ of order $p^3$ such 
that 
\begin{equation}\label{eq:U12}
\A_{U_1} \cong \D_p~\text{and}~
\A_{U_2} \cong \Z C_p \wr \Z C_p^2.
\end{equation}
Moreover, each basic set of $\A$ 
outside $U_1 \cup U_2$ is a $V$-coset, 
where $V$ is the unique $\A$-subgroup of $G$ of order $p^2$ 
contained in $U$. 
\end{lemm}

\begin{proof}
Let $(a,b,c)$ be an $\A_U$-basis, and so $\{a\}, bA, c(A\times B) \in \cS(\A_U)$, where $A=\langle a \rangle$ and 
$B=\langle b \rangle$. Clearly, $V=A \times B$.

Suppose for the moment that there exists 
$X \in \cS(\A)$ outside $U$ with $|X|=1$. In this case $X=\{x\}$ for some $x\in G \setminus U$.  Clearly, $\A_{\langle X\rangle}=\Z \langle X \rangle$. Lemma~\ref{tenspr} implies that $\A=\A_U \otimes \A_{\langle X \rangle}$. Let $\varphi \in \aut_U(\A_U)$. Define 
$\psi \in \aut(G)$ in the following way: 
$$
\psi^U=\varphi,~x^\psi=x.
$$ 
Then $\psi \in \aut_G(\A)$. We obtain that $\aut_{G}(\A)^U \geq \aut_U(\A_U)$, and therefore, $\aut_G(\A)^U=\aut_U(\A_U)$. 
Thus, we may assume that $|X|>1$ for every $X \in \cS(\A)$ outside 
$U$. This implies that $A$ is the unique $\A$-subgroup of 
order~$p$.

From now on let $X \in \cS(\A)$ be a basic set such that $X \not\subset U$ and $|\rad(X)|=1$. 

The $S$-ring $\A$ is decomposable. So $\A$ is the $(\bar{U}/\bar{L})$-wreath product for some $\A$-section $\bar{U}/\bar{L}$ with 
$\bar{U} < G$ and $|\bar{L}| > 1$. 
Note that, $X \subseteq \bar{U}$ because every basic set of $\A$ outside $\bar{U}$ has nontrivial radical, whereas $|\rad(X)|=1$. 
Since $\langle X \rangle \le \bar{U} < G$, we have 
$|\langle X \rangle| \leq |\bar{U}| \leq p^3$. In view of the previous paragraph, we may assume that $|X|>1$. Now, the description of all 
$p$-S-rings over an elementary abelian group of rank at most~$3$ given in Lemma~\ref{p3} implies that $p \neq 2$, $|X|=p$,  
$\langle X \rangle=\bar{U}$, $|\bar{U}|=p^3$, and 
$\A_{\bar{U}} \cong \D_p$, the S-ring in row no.~6 in Table~1. Therefore, choosing $U_1$ to be $\bar{U}$, the first part of 
Eq.~\eqref{eq:U12} holds. 
 
Since $V$ is the unique $\A$-subgroup of $G$ of order $p^2$ 
contained in $U$, it follows that $V=U \cap U_1$. Therefore, we may assume that
$$
X=\big\{ xb^ia^{\frac{i(i-1)}{2}}:~i \in \{0,\ldots,p-1\} \big\},
$$
for some $x\in G\setminus U$.

Let $Y \in \cS(\A)$ outside $U \cup U_1$. Assume that $|\rad(Y)|=1$. Then $Y \subseteq U_1$ because every basic set outside $U_1$ has  nontrivial radical, a contradiction.  So $|\rad(Y)|>1$. Since $A$ is the unique $\A$-subgroup of order $p$, we conclude that 
$A \leq \rad(Y)$. 

Let $\pi: G \rightarrow G/A$ be the canonical epimorphism. Consider 
the S-ring $\A_{G/A}$ over the group $G/A$ of order $p^3$. Note that,  $|\pi(X)|=|\pi(c V)|=p$ and $\rad(\pi(X))=\rad(\pi(c V))=\pi(V)$. Now, the description of all $p$-S-rings over $C_p^3$ given in Table~1 
implies that $\A_{G/A}$ is isomorphic to 
$\Z C_p \wr \Z C_p^2$ or $(\Z C_p \wr \Z C_p) \otimes \Z C_p$. 

Assume first that $\A_{G/A} \cong \Z C_p \wr \Z C_p^2$. 
Then $|\pi(Y)|=p$ and $\rad(\pi(Y))=\pi(V)$. Since $A \leq \rad(Y)$, we conclude that $Y$ is a $V$-coset. Thus we proved that every basic set of $\A$ outside $U \cup U_1$ is a $V$-coset. A straightforward check shows that 
$\aut_G(\A)$ contains the following subgroup:
$$
M=\big\{ \varphi \in \aut(G):~(a,b,c,x)^\varphi=
(a,ba^i,ca^jb^k,xb^ia^{\frac{i(i-1)}{2}}),\, i, j, k \in \{0,\ldots,p-1\} 
\big\}.
$$
Therefore, $|\aut_G(\A)| \geq |M|=p^3$. 

Suppose that $\varphi \in \aut_{G}(\A)$ acts trivially on $U$. Then 
$\varphi^{U_1} \in \aut_{U_1}(\A_{U_1})$ acts trivially on 
$V$. The S-ring $\A_{U_1}$ is Cayley minimal by Lemma~\ref{p3caymin}. So $|\aut_{U_1}(\A_{U_1})|=p$ and hence 
$b A$ is a regular orbit of 
$\aut_{U_1}(\A_{U_1})$. Therefore, $\varphi^{U_1}=\id_{U_1}$  because it acts trivially on $b A$. This yields that $\varphi$ acts trivially on $\langle U, U_1 \rangle=G$ and hence $|\aut_{G}(\A)^U|=|\aut_{G}(\A)|$. Using this and Lemma~\ref{p3notcaymin}, we conclude that 
$|\aut_{G}(\A)^U| \geq p^3=|\aut_U(\A_U)|$, and therefore,    
$\aut_{G}(\A)^U=\aut_U(\A_U)$.

Let $\A_{G/A} \cong (\Z C_p \wr \Z C_p) 
\otimes \Z C_p$. In this case $\O_\theta(\A_{G/A})$ has order 
$p^2$.  Choose $U_2$ to be the $\A$-subgroup of $G$ of order 
$p^3$ such that $\pi(U_2)=\O_\theta(\A_{G/A})$. Thus 
$\A_{U_2} \cong \Z C_p \wr \Z C_p^2$, so the second part 
of Eq.~\eqref{eq:U12} holds. 

Finally, if $Y \in \cS(\A)$ and $Y$ is outside $U_1 \cup U_2$, then 
$\pi(Y)=\pi(V)$, so $Y$ is a $V$-coset. This completes the proof of 
the lemma.
\end{proof}

\section{S-rings over an abelian group of non-powerful order}

A number $n$ is called \emph{powerful} if $p^2$ divides $n$ for every prime divisor $p$ of $n$. 
Wielandt~\cite{W} showed that, if $G$ is an abelian group of 
composite order with at least one cyclic Sylow 
subgroup, then the only primitive S-ring over $G$ is the one of 
rank~$2$ (see also \cite[Theorem~25.4]{Wi}). In particular, the following statement holds:

\begin{lemm}\label{primitive}
{\rm (cf. \cite{W})}
If $G$ is an abelian group of non-powerful composite order, then the only primitive S-ring over $G$ is the one of rank~$2$.
\end{lemm}

In the remaining lemmas the group $G=P \times Q$, where $P$ is an abelian group and $Q \cong C_q$ with prime $q$ coprime to $|P|$.  
Clearly, $|G|$ is non-powerful. Furthermore, $\A$ is an S-ring over 
$G$, $\bar{P}$ is the unique maximal $\A$-subgroup contained in $P$, 
and $\bar{Q}$ is the least $\A$-subgroup containing $Q$. 
Note that, $\bar{P} \bar{Q}$ is an $\A$-subgroup.

\begin{lemm}\label{nonpower1}
{\rm (\cite[Proposition~13]{MS})} 
The group $\bar{P}$ is a maximal $\A_{\bar{P} \bar{Q}}$-subgroup.
\end{lemm}

 The following lemma can be retrieved from the proof 
of \cite[Proposition~14]{MS}.

\begin{lemm}\label{nonpower2}
{\rm (cf.\ \cite[Proposition~14]{MS})}
If $\bar{P} \ne (\bar{P} \bar{Q})_{q'}$, the Hall $q'$-subgroup of $\bar{P} \bar{Q}$, 
then $\A_{\bar{P} \bar{Q}}=\A_{\bar{P}} \star \A_{\bar{Q}}$.
\end{lemm}

\begin{lemm}\label{nonpower3}
{\rm (\cite[Proposition 15]{MS})}
If $\A_{\bar{P} \bar{Q}/\bar{P}} \cong \Z C_q$, then 
$\A_{\bar{P} \bar{Q}}=\A_{\bar{P}} \star \A_{\bar{Q}}$. 
\end{lemm}

\begin{lemm}\label{nonpower4}
{\rm (\cite[Lemma~6.2]{EKP})}
If $\bar{P} < P$, then one of the following statements holds: 
\begin{enumerate}[{\rm (1)}]
\item $\A=\A_{\bar{P}} \wr \A_{G/\bar{P}}$ with 
$\rk(\A_{G/\bar{P}})=2$;
\item $\A=\A_{\bar{P} \bar{Q}} \wr_S \A_{G/\bar{Q}}$, where 
$S=\bar{P} \bar{Q}/\bar{Q}$ and $\bar{Q} < G$.
\end{enumerate}
\end{lemm}

Let $\pr_P$ denote the \emph{projection} from 
$G=P \times Q$ to $P$. 

\begin{lemm}\label{nonpower5}
{\rm (\cite[Lemma~6.1]{EKP})}
For each basic set $X \in \cS(\A)$, $\pr_P(X) \setminus X$ is an $\A$-set. 
\end{lemm}

\begin{lemm}\label{nonpower6}
{\rm (cf.\ \cite[Statement~1 of Lemma~2.3]{EKP})}
If both $P$ and $Q$ are $\A$-subgroups, then 
for each basic set $X \in \cS(\A)$, $\pr_P(X) \in \cS(\A)$. 
\end{lemm}

\section{Proof of~Theorem~\ref{main}}

In view of Lemma~\ref{cimin}, it is sufficient to prove the following 
theorem:

\begin{theo}\label{main2}
Let $G=P \times Q$, where $P \cong C_p^4$ and $Q \cong  C_q$ 
for distinct primes $p$ and $q$, and let 
$K \in \Sup_2^{\rm min}(G_\r)$. 
Then $V(K, G)$ is a $\CI$-S-ring.
\end{theo}

The proof of the theorem consists of two main steps.  
It will be shown first that the S-ring 
$V(K, G)$ in Theorem~\ref{main2} is $\CI$, unless it is a 
generalized wreath product with certain properties 
(Lemmas~\ref{step1}-\ref{step3}).
Then the latter generalized wreath product 
will be shown to be $\CI$ after an exhaustive computation.

\begin{lemm}\label{step1}
With the notations in Theorem~\ref{main2},  let 
$\A=V(K, G)$. Then one of the following conditions holds:
\begin{enumerate}[{\rm (1)}]
\item $\A$  is a $\CI$-S-ring.  
\item $\A= \A_{\bar{P}Q}\, \wr_S\, \A_Q$,
where $S=\bar{P}Q/Q$, $\bar{P}$ is the maximal $\A$-sugroup contained in $P$ and $\bar{P} < P$. 
\end{enumerate}
\end{lemm}

\begin{proof} 
Assume that $\A$ is a non-$\CI$-S-ring. We have to 
show that case (2) holds. 

Let $\bar{Q}$ be the least $\A$-subgroup containing $Q$.
Assume first that $\bar{P}=P$. Then $\A_{G/P}$ is a $q$-$S$-ring by Lemma~\ref{minpring}. 
Lemma~\ref{p3} implies that $\A_{G/P} \cong \Z C_q$. Clearly, 
$G=P\bar{Q}$. It follows from Lemma~\ref{nonpower3} that 
$\A=\A_P \star \A_{\bar{Q}}$. 
Then $\A$ is a 
$\CI$-S-ring by Remark~\ref{cisrings} and Lemma~\ref{cistar}, 
contradicting our assumption that $\A$ is a non-$\CI$-S-ring. 

Thus $\bar{P} < P$, and Lemma~\ref{nonpower4} can be applied to 
$\A$. Suppose that case~(1) of Lemma~\ref{nonpower4} holds for $\A$, that is, 
$$
\A= \A_{\bar{P}}\, \wr\, \A_{G/\bar{P}},
$$ 
where $\rk(\A_{G/\bar{P}})=2$. If $\bar{P}$ is trivial, then
$\rk(\A)=2$ and $\A$ is a $\CI$-S-ring. If $\bar{P}$ is nontrivial, then 
$\A$ is a $\CI$-S-ring by Remark~\ref{cisrings} and 
Lemma~\ref{cigwr}, which is impossible.

Thus case~(2) of Lemma~\ref{nonpower4} holds for 
$\A$, that is,  
$$
\A=\A_{\bar{P} \bar{Q}}\, \wr_S\, \A_{G/\bar{Q}},
$$ 
where $S=\bar{P} \bar{Q}/\bar{Q}$ and $\bar{Q} < G$. 
Notice that, we are done if we show that 
$\bar{Q}=Q$.

Assume for the moment that $\bar{P} \bar{Q}=G$. 
Then $\bar{P} \ne (\bar{P} \bar{Q})_{q'}=P$, hence by 
Lemma~\ref{nonpower2}, $\A=\A_{\bar{P}} \star \A_{\bar{Q}}$.
Since $\bar{P} < G$ and $\bar{Q} < G$, the S-ring $\A$ is a 
$\CI$-S-ring by Remark~\ref{cisrings} and Lemma~\ref{cistar}. 

Thus $\bar{P} \bar{Q} < G$, and therefore,  
$\A$ is the nontrivial $S$-wreath product. The group $G/\bar{Q}$ is a $p$-group of order at most $p^4$ because $\bar{Q} \geq Q$. Lemma~\ref{minpring} yields that 
$\A_{G/\bar{Q}}$ and $\A_S$ are $p$-S-rings. If 
$|G/\bar{Q}| \leq p^3$ then the S-rings 
$\A_{G/\bar{Q}}$ and $\A_S$ are cyclotomic by Lemma~\ref{p3}. Since $|S| \leq p^2$, the S-ring $\A_S$ is Cayley minimal by Lemma~\ref{p3caymin}. It follows from Lemma~\ref{cicayleymin} that 
$\A$ is a $\CI$-S-ring. 
Thus, $|G/\bar{Q}|=p^4$ and hence $\bar{Q}=Q$. 
\end{proof}

\begin{lemm}\label{step2}
With the notations in Lemma~\ref{step1}, suppose that case (2) 
holds. Then one of the following conditions holds: 
\begin{enumerate}[{\rm (1)}]
\item $\A$  is a $\CI$-S-ring.  
\item $\A_{G/Q}$ is decomposable and 
$\A_S \cong \Z_p C_p \wr \Z_p C_p \wr \Z_p C_p$.
\end{enumerate}
\end{lemm}

\begin{proof}
Suppose that $\A_{G/Q}$ is indecomposable. Then $\A_{G/Q}$ is 
$2$-minimal by Lemma~\ref{p4minimal}. Let 
$K=\aut(\A)_{G/Q}G_{\rm right}$ and $\B=V(K,\, G)$. 

We show next that $\B$ is a $\CI$-S-ring. 
The definition of $\B$ implies that every $\A$-subgroup is also 
a $\B$-subgroup. Let $\tilde{P}$ be the largest 
$\B$-subgroup contained in $P$. Then $\bar{P} \le \tilde{P}$, 
and the section $\tilde{S}:=\tilde{P}Q/Q$ is a 
$\B$-section. From the definition of $K$ it follows that every basic set of $\B$ is contained in a $Q$-coset, implying that 
$\B_{\tilde{P}}=\Z \tilde{P}$. Thus, if $\tilde{P}=P$, 
then $\B=\B_Q \otimes \B_{\tilde{P}}=
\B_Q \star \B_{\tilde{P}}$ by Lemma~\ref{tenspr}, 
and $\B$ is a $\CI$-S-ring by Remark~\ref{cisrings} and 
Lemma~\ref{cistar}. 
Let $\tilde{P} < P$. Then, by Lemma~\ref{nonpower4}, 
$\B$ is the nontrivial $\tilde{S}$-wreath poduct.   
Since every basic set of $\B$ is contained in a $Q$-coset, 
$\B_{G/Q}=\Z (G/Q)$ and hence $\B_{\tilde{S}}=\Z \tilde{S}$.
By Lemma~\ref{cicayleymin}, $\B$ is a $\CI$-S-ring. 

The S-ring $\A_{G/Q}$ is a $\CI$-S-ring by Remark~\ref{cisrings}. 
All conditions of Lemma~\ref{minnorm} hold for $\A$, $\A_{G/Q}$ and $\B$ and hence $\A$ is a $\CI$-S-ring.

Now, suppose that $\A_{G/Q}$ is decomposable. 
Then $\A_{G/Q}$ is cyclotomic by Lemma~\ref{p4decomp}. If $\A_S$ is Cayley minimal then we are done by Lemma~\ref{cicayleymin}. So we may assume that 
$\A_S$ is not Cayley minimal. Since $|S| \leq p^3$, Lemma~\ref{p3caymin} implies that $|S|=p^3$ and 
$\A_S \cong \Z C_p \wr \Z C_p \wr \Z C_p$. 
\end{proof}

\begin{lemm}\label{step3}
With the notations in Lemma~\ref{step2}, suppose that case (2) 
holds. Then one of the following conditions holds: 
\begin{enumerate}[{\rm (1)}]
\item $\A$  is a $\CI$-S-ring.  
\item $\A_{G/Q}$ has a basic set outside $S$ with trivial radical.
\end{enumerate}
\end{lemm}

\begin{proof}
The S-ring $\A_S \cong \Z C_p \wr \Z C_p \wr \Z C_p$, 
hence there exists a unique minimal $\A_S$-subgroup in $S$, 
say $A$, which has order $p$.
In view of Lemma~\ref{minpring}, we have 
$\A_{(G/Q)/S} \cong \Z C_p$. This means that every basic set of 
$\A_{G/Q}$ outside $S$ is contained in an $S$-coset. 
So $\rad(X)$ is an $\A_S$-subgroup for every $X \in \cS(\A_{G/Q})$ outside $S$. If $|\rad(X)|>1$ for every $X \in \cS(\A_{G/Q})$ outside $S$, then $\A_{G/Q}$ is the $S/A$-wreath product because $A$ is the least $\A_S$-subgroup. Therefore, $\A$ is the $\bar{P}Q/\pi^{-1}(A)$-wreath product, where $\pi: G \rightarrow G/Q$ is the canonical epimorphism. Note that, $|\pi^{-1}(A)|=pq$ and hence $|G/\pi^{-1}(A)|=p^3$ and $|\bar{P}Q/\pi^{-1}(A)|=p^2$. The S-rings $\A_{G/\pi^{-1}(A)}$ and $\A_{\bar{P}Q/\pi^{-1}(A)}$ are cyclotomic by Lemma~\ref{p3} and the S-ring $\A_{\bar{P}Q/\pi^{-1}(A)}$ is Cayley minimal by Lemma~\ref{p3caymin}. Thus, $\A$ is a $\CI$-S-ring by Lemma~\ref{cicayleymin}.
\end{proof}

We need one more 
technical lemma. 

\begin{lemm}\label{technical}
Let $\A$ be a schurian S-ring over a group 
$H \cong C_p^3 \times C_q$, where 
$p$ and $q$ are distinct primes, and let $H_p$ and $H_q$ be 
the Sylow $p$- and $q$-subgroup of $H$, respectively. 
Suppose that $H_p$ and $H_q$ are $\A$-subgroups, 
$\A_{H_p} \cong \Z C_p \wr \Z C_p \wr \Z C_p$, 
and at least one of the following conditions holds: 
\begin{enumerate}[{\rm (1)}]
\item $\aut(\A)^{H_q}$ is non-solvable.
\item $V((\aut(\A)_{H/H_p})^{H_p},H_p)$ is decomposable. 
\end{enumerate}
Suppose, in addition, that $f\in \iso_e(\A)$ such that 
$$
a^f=a,~b^f=b,~c^f \in \langle a \rangle c~
\text{and}~d^f=d,
$$
where $(a,b,c)$ is an $\A_{H_p}$-basis and 
$H_q=\langle d \rangle$.
Then $X^f=X$ for each $X \in \cS(\A)$.
\end{lemm}

\begin{proof}
Let $A_1=\langle a \rangle$, $A_2=\langle a, b \rangle$ and 
$A_3=\langle a, b, c \rangle$. Then each $A_i$ is an 
$\A$-subgroup, and $A_3=H_p$. Choose an arbitrary basic set $X \in \cS(\A)$. 

It is easy to check that $X^f=X$ holds whenever 
$X \subset H_p$.  Suppose that $X \subset H_q$. 
If $\A_{H_q}$ is of rank $2$, then $X^f=X$ holds trivially. If 
$\A_{H_q}$ is of rank larger than $2$, then $\A_{H_q}$ is normal 
by \cite[Theorem~4.1]{EP2}. Since $\A_{H_q}$ is a $\CI$-S-ring, see Remark~\ref{cisrings},  
and $d^f=d$, Lemma~\ref{cinormal}(ii) can be applied to $\A$, 
$H_q$ and $f$, and this results in 
$f^{H_q}=\id_{H_q}$, and so $X^f=X$.  
Notice that, all these yield that, for every 
$i \in \{1, 2, 3\}$, 
\begin{equation}\label{eq:imp}
\A_{A_i H_q}=\A_{A_i} \otimes \A_{H_q} \implies 
X^f=X~\text{if}~X \subset A_i H_q.
\end{equation}

Let $K=\aut(\A)$, $L=K_{H/H_p}$ and $M=K_{H/H_q}$ (for the definition of the latter 
two groups, see Eq.~\eqref{eq:kernel}). 

Let $x  \in H_q$. Then for any $\gamma \in L$, 
$e^\gamma=e$ if and only if $x^\gamma=x$.  Thus, $L_e=L_x$, 
and so $L$ acts equivalently on $H_p$ and $H_px$. This implies that 
$L$ acts faithfully on $H_px$, in particular, $L \cong L^{H_px}$. 
One obtains in exactly the same way that, for every $y \in H_p$,  
$M$ acts equivalently on $H_q$ and 
$H_qy$ and $M \cong M^{H_qy}$.

Assume first that condition~(1) holds, that is, 
$K^{H_q}$ is non-solvable. Then $K^{H_q}$ acts $2$-transitively on 
$H_q$.
The group $(H_q)_\r \le M^{H_q}$. 
If $M^{H_q}$ is solvable, then $(H_q)_\r$ is characteristic in 
$M^{H_q}$, and as $M^{H_q} \lhd K^{H_q}$, it follows that 
$(H_q)_\r \lhd K^{H_q}$. This, however, contradicts that 
$K^{H_q}$ is non-solvable, and thus $M^{H_q}$ is a 
non-solvable group. This implies that $M$ acts $2$-transitively 
on every coset ${H_q}x$, $x \in H_p$. Thus, $H_qx \setminus \{x\}$ is an orbit under $M_x=M_e$, showing that $\A=\A_{H_p} \otimes 
\A_{H_q}$. 
Then we are done by Eq.~\eqref{eq:imp}.

In the rest of the proof we assume that $K^{H_q}$ is a solvable group. 
Notice that,  since $M \cong M^{H_q}$, it follows in turn that, $M$ is solvable,  $(H_q)_r$ is characteristic in $M$, and as 
$M \lhd K$, we have $(H_q)_r \lhd K$.

To settle the lemma, we have to show that $X^f=X$ for each 
$X \in \cS(\A)$ provided that condition~(2) holds, that is, 
$V(L^{H_p},H_p)$ is decomposable. 
Let $\B=V(L^{H_p},H_p)$.
Since $\A_{H_p} \cong \Z C_p \wr  \Z C_p \wr  \Z C_p$, 
$\aut(\A_{H_p})$ is a $p$-group, and hence so is $L^{H_p}$. 
We find that $\B$ is isomorphic to one of the S-rings in rows nos.~2--5 in Table~1. We consider below all the possibilities step by step.
Recall that, $X$ is an arbitrary basic set of $\A$.
\medskip

\noindent{\bf Case~1.} $\B \cong \Z C_p^2 \wr \Z C_p$.
\medskip

In this case $\rad(X) \ge A_2$ if  
$X \not\subset A_2H_q$, and so $X^f=X$ for such $X$.
The S-ring $\B_{A_2}=\Z A_2$, hence 
$L^{A_2}=(A_2)_\r$. 
Repeating the argument used by showing that 
$L \cong L^{H_p}$, 
one can prove that $L^{A_2H_q} \cong L^{A_2}$.  This then yields 
$((A_2)_r)^{A_2H_q}=L^{A_2H_q} \lhd K^{A_2H_q}$. On the other hand, $(H_q)_r \lhd K$ also holds, and thus we find that 
$(A_2H_q)_\r \lhd K^{A_2H_q}$. By Remark~\ref{cisrings}, 
$\A$ is a $\CI$-S-ring, and all conditions in  
Lemma~\ref{cinormal}(i) hold for $\A, A_2H_q, f$ and 
the generating set $\{a,b,d\}$.   
We find $f^{A_2H_q}=\id_{A_2H_q}$, and so $X^f=X$ for 
$X \subset A_2H_q$ as well.
\medskip

\noindent{\bf Case~2.} $\B \cong \Z C_p \wr \Z C_p^2$.
\medskip

In this case $\rad(X) \ge A_1$ if  
$X \not\subset A_1Q$. Thus, $\A_{A_2H_q}=\A_{A_2} \otimes \A_{H_q}$. 
By Eq.~\eqref{eq:imp}, $X^f=X$ if $X \subset A_2H_q$. 
The S-ring $\B_{H_p/A_1}=\Z (H_p/A_1)$, hence 
$L^{H_p/A_1}=(H_p/A_1)_{\rm right}$.
It can be shown that $L^{H/A_1} \cong L^{H_p/A_1}$.
This then yields 
$((H_p/A_1)_r)^{H/A_1}=L^{H/A_1} \lhd K^{H/A_1}$, 
and since $Q_r \lhd K$ also holds, we obtain that 
$(H/A_1)_\r \lhd K^{H/A_1}$. 
Lemma~\ref{cinormal}(i) can be applied to 
$\A$, $H/A_1$, $f$ and the generating set 
$\{A_1b, A_1c, A_1d\}$. We find 
$f^{H/A_1}=\id_{H/A_1}$, and so $X^f=X$ for 
$X \not\subset A_2H_q$ as well.
\medskip

\noindent{\bf Case~3.} $\B \cong (\Z C_p \wr \Z C_p) \otimes \Z C_p$.
\medskip

In this case $|\O_\theta(\B)|=p^2$. 
Assume for the moment that $\O_\theta(\B) \ne A_2$, and let 
$\{u\} \in \cS(\B)$ for some $u \notin A_2$. Choose 
$X \in \cS(\A)$ so that $\pr_{H_p}(x)=u$ for some 
$x \in X$ (recall that, $\pr_{H_p}$ denotes the projection from 
$H=H_p \times H_q$ to $H_p$). 
Then $H_px \cap X=\{x\}$. For otherwise, there is 
$\gamma \in K_e$ such that $x^\gamma \ne x$ and 
$x^\gamma \in H_px$. Since $K^{H_q}$ is solvable, this implies in turn that, $\gamma \in L_e$, 
$u^\gamma \ne u$ and $u^\gamma \in H_p$. This contradicts that 
$\{u\} \in \cS(\B)$. We obtain that every $H_p$-coset intersects 
$X$ in at most one element. 
Let $v \in A_2u$ be an arbitrary element.
By Lemma~\ref{nonpower6}, 
$\pr_{H_p}(X)=A_2u$, $\pr_{H_p}(y)=v$ for some $y \in X$. 
If $v \notin \O_\theta(\B)$, then there exists  
$\gamma \in L_e$ such that $v^\gamma \ne v$. But then 
$y \ne y^\gamma \in X^\gamma=X$, contradicting 
that $H_py \cap X=\{y\}$. We obtain that $A_2u \subset 
\O_\theta(\B)$, so $\B=\Z H_p$, a contradiction. 

Then $\O_\theta(\B)=A_2$, hence $\B_{A_2}=\Z A_2$. 
One can repeat the argument used by the 1st case and derive 
that $X^f=X$ if $X \subset A_2H_q$. 

Suppose that $X \not\subset A_2H_q$.
Each basic set of $\B$ outside $A_2$ is equal to an $A^*$-coset 
for some subgroup $A^* < A_2$ of order $p$, in particular, 
$A^* \le \rad(X)$. If $A^* \ne A_1$, 
then $A_2 \le \rad(X)$ because $A_1$ is the unique $\A$-subgroup 
of order $p$, and it follows from this that 
$X^f=X$. If $A^*=A_1$, then $\B_{H_p/A_1}=\Z (H_p/A_1)$. 
In this case one can procced as by the 2nd case and derive 
that $X^f=X$. 
\medskip

\noindent{\bf Case~4.} $\B \cong \Z C_p \wr \Z C_p \wr \Z C_p$. 
\medskip

In this case $\A=\A_{H_p} \otimes \A_{H_q}$, hence 
we are done by Eq.~\eqref{eq:imp}. 
\end{proof}

We are ready to prove the main result of this paper.

\begin{proof}[Proof of Theorem~\ref{main2}.]
Let $\A=V(K, G)$ be the S-ring given in Theorem~\ref{main2}.
Due to Lemma~\ref{step1} we may assume that 
$$
\A= \A_{\bar{P}Q}\, \wr_S\, \A_Q,
$$
where $S=\bar{P}Q/Q$, $\bar{P}$ is the maximal $\A$-sugroup contained in $P$ and $\bar{P} < P$.
It follows from Lemmas~\ref{step2} and \ref{step3}
that we may further assume that 
$\A_{G/Q}$ is decomposable,  
$\A_S \cong \Z_p C_p \wr \Z_p C_p \wr \Z_p C_p$, and 
$\A_{G/Q}$ has a basic set outside $S$ with 
trivial radical. By the latter conditions Lemma~\ref{p3exceptcase} 
can be applied to $\A_{G/Q}$ and $S$.

If $\aut_{G/Q}(\A_{G/Q})^S=\aut_S(\A_S)$, then $\A$ is a $\CI$-S-ring by Lemma~\ref{cigwr}, and thus we may assume that 
$\aut_{G/Q}(\A_{G/Q})^S \ne \aut_S(\A_S)$. 

Now, according to Lemma~\ref{p3exceptcase}, $p>2$ and 
there exist $\A$-subgroups $U_1$ and $U_2$ of $G$ of 
order $p^3q$ such that $U_1, U_2 \ne P_1Q$, 
\begin{equation}\label{eq:U13}
\A_{U_1/Q}\cong \D_p~\text{and}~
\A_{U_2/Q} \cong \Z C_p \wr \Z C_p^2.
\end{equation}
Furthermore, 
each basic set of $\A_{G/Q}$ outside $U_1/Q \cup U_2/Q$ is a $V$-coset, where $V$ is the unique 
$\A_{G/Q}$-subgroup of $G$ of $p^2$ contained in $S$. 

Choose an arbitrary isomorphism $f \in \iso_e(\A)$. 
We complete the proof of the theorem by 
finding an automorphism $\varphi \in \aut(G)$ such that
\begin{equation}\label{eq:XfeqXphi}
X^{f \varphi^{-1}}=X~\text{for each}~X \in \cS(\A).
\end{equation}
Indeed, then $f \in \aut(\A)_e\aut(G)$, which 
implies $\iso_e(\A)=\aut(\A)_e\aut(G)$, and so $\A$ is a 
$\CI$-S-ring by Lemma~\ref{trans}. 

Fix elements $a_1, \ldots, a_6$ of $G$ such that 
$(a_1, a_2, a_3)$ is an $\A_{\bar{P}}$-basis,  
$$
a_4 \in (P \cap U_1) \setminus \bar{P},~  
a_5 \in (P \cap U_2) \setminus \bar{P}~\text{and}~ 
Q=\langle a_6 \rangle.
$$ 
Then $\langle a_6^f \rangle=Q$ and 
$\langle a_1^f, a_2^f, a_3^f \rangle \le P$ of order $p^3$. 
Let $P_1=\langle a_1 \rangle$, $P_2=\langle a_1, a_2 \rangle$ 
and $U=\bar{P}Q$. 

Notice that, both $\bar{P}$ and $Q$ are $\A_U$-subgroups and $\A_{\bar{P}} \cong \Z C_p \wr \Z C_p \wr \Z C_p$. 
Thus $\A_U$ satisfies the first three conditions in Lemma~\ref{technical}. The rest of the proof is divided into two cases 
depending whether also at least one of conditions (1) and (2) 
in Lemma~\ref{technical} holds.
\medskip

\noindent{\bf Case 1.}\ $\aut(\A_U)^Q$ is non-solvable or 
$V((\aut(\A_U)_{U/\bar{P}})^{\bar{P}},\bar{P})$ 
is decomposable.
\medskip

It is easy to see that there exists an automorphis $\varphi \in \aut(G)$ such that 
$$
a_i^\varphi=a_i^f~\text{for each}~i \in \{1,2,6\},~Qa_4^\varphi=Qa_4^f~\text{and}~Qa_5^\varphi=Qa_5^f. 
$$
We claim that Eq.~\eqref{eq:XfeqXphi} holds for $f_1:=f \varphi^{-1}$.

Let $X \in \cS(\A)$ such that 
$X \not\subset U \cup U_1 \cup U_2$. Notice that $X$ is a 
$P_2Q$-coset. This implies that $\A_{G/P_2Q} \cong \Z C_p^2$, 
in particular, $\A_{G/P_2Q}$ is a normal $\CI$-S-ring. 
Then, Lemma~\ref{cinormal}(ii) can 
be applied to $\A$, $G/P_2Q$, $f_1$ and the generating set 
$\{P_2Qa_4, P_2Qa_5\}$. We obtain that 
$f_1^{G/P_2 Q}=\id_{P_2Q}$, and so $X^{f_1}=X$.

Let $X \in \cS(\A)$ such that $X \subset U_1 \setminus U$. 
By Remark~\ref{cisrings} and Eq.~\eqref{eq:U13}, $\A_{U_1/Q}$ is a normal $\CI$-S-ring. Lemma
~\ref{cinormal}(ii) can be applied to $\A$, $U_1/Q$, $f_1$ and the 
generating set $\{Qa_1, Qa_2$, $Qa_4\}$. Then,  
$f_1^{U_1/Q}=\id_{U_1/Q}$. Using also that 
$\A$ is a $U/Q$-wreath product, we obtain that $X^{f_1}=X$.

Let $X \in \cS(\A)$ such that $X \subset U_2 \setminus U$. 
By Eq.~\eqref{eq:U13}, $\A_{U_2/P_1Q} \cong \Z C_p^2$, and hence 
Lemma~\ref{cinormal}(ii) can be applied to the $\A$, $U_2/P_1Q$, 
$f_1$ and the generating set $\{P_1Qa_2, P_1Qa_5\}$. 
Then, $f_1^{U_2/P_1Q}=\id_{U_2/P_1Q}$, and using also that 
$\A_{U_2}$ is a $(U \cap U_2)/P_1Q$-wreath product, we obtain 
that $X^{f_1}=X$.

Finally, let $X \in \cS(\A)$ such that $X \subset U$. 
The S-ring $\A_{G/P_1Q} \cong (\Z C_p \wr \Z C_p) \otimes \Z C_p$ 
with $\O_\theta(\A_{G/P_1Q})=U_2/P_1Q$.
It has been shown above that 
$$
f_1^{U_1/P_1Q}=
\id_{U_1/P_1Q}~\text{and}~f_1^{U_2/P_1Q}=\id_{U_2/P_1Q}.
$$ 

An arbitrary element of $G/P_1Q$ is in the form 
$P_1Qxy$, where $x \in U_1$ and $y \in U_2$. 
Since $\{P_1Qy\} \in \cS(\A_{G/P_1Q})$, 
$$
(P_1Qxy)^{f_1^{G/P_1Q}}=(P_1Qx)^{f_1^{G/P_1Q}}
(P_1Qy)^{f_1^{G/P_1Q}}=P_1Q xy.
$$
Thus, $f_1^{G/P_1Q}=\id_{G/P_1Q}$, in particular,   
$a_3^{f_1} \in P_1Q a_3$. 
On the other hand, $a_3^{f_1} \in \bar{P}$, hence 
$a_3^{f_1} \in P_1a_3$.  Then all conditions in Lemma~\ref{technical} 
hold for $\A_U$ and $f_1^U$, and this results in 
$X^{f_1}=X$. 
\medskip

\noindent{\bf Case 2.}\ $\aut(\A_U)^Q$ is solvable and 
$V((\aut(\A_U)_{U/\bar{P}})^{\bar{P}},\bar{P})$ is indecomposable.
\medskip

In this case, we choose an automorphism $\varphi \in \aut(G)$ 
such that 
$$
a_i^\varphi=a_i^f~\text{for each}~i \in \{1,2,3,6\}~
\text{and}~Qa_4^\varphi=Qa_4^f. 
$$
We claim that Eq.~\eqref{eq:XfeqXphi} holds for 
$f_1:=f \varphi^{-1}$.

We show first that $\A_U$ is normal.  
Using that $\aut(\A_U)^Q$ is solvable, we have shown 
in the proof of Lemma~\ref{technical} that 
$(Q_r)^U \lhd \aut(\A_U)$. Also, since 
$V((\aut(\A_U)_{U/\bar{P}})^{\bar{P}},\bar{P})$ is indecomposable, 
$\bar{P}_\r$ is the unique regular subgroup of 
$(\aut(\A_U)_{U/\bar{P}})^{\bar{P}}$ isomorphic to $C_p^3$. 
Thus, for every $\gamma \in \aut(\A_U)$, 
$((\bar{P}_r)^U)^\gamma)^{\bar{P}}=\bar{P}_\r=
(P_r)^{\bar{P}}$. 
This implies that $((\bar{P}_r)^U)^\gamma=(\bar{P}_r)^U$ because  
$\aut(\A_U)_{U/\bar{P}}$ acts faithfully on $\bar{P}$, see the 
proof of Lemma~\ref{technical}.  
We obtain that $(\bar{P}_r)^U \lhd \aut(\A_U)$. 
Since $(Q_r)^U \lhd \aut(\A_U)$ and 
$U_\r=(\bar{P}_r)^U (Q_r)^U$, it follows that $\A_U$ is indeed normal.  

Let $L=\aut_U(\A_U)^{U/Q}$ and 
$M=\aut_{U/Q}(\A_{U/Q})$. Clearly, $L \le M$. 
Write $\bar{x}$ for the element 
$Qx \in U/Q$. The S-ring 
$\A_{U/Q} \cong \Z C_p \wr \Z C_p \wr \Z C_p$ and 
$\{\bar{a}_1,\bar{a}_2,\bar{a}_3\}$ is an $\A_{U/Q}$-basis. 
These yield that 
$$
M=\big\{\, \phi \in \aut(U/Q):~(\bar{a}_1,\bar{a}_2,\bar{a}_3)^\phi=
(\bar{a}_1,\bar{a}_2\bar{a}_1^{\, i},\bar{a}_3\bar{a}_2^{\, j} 
\bar{a}_1^{\, k}),~i, j, k \in \{0,\ldots,p-1\}\, 
\big\}. 
$$

Since $\A_U$ is normal, 
$L \approx_{\rm Cay} M$, hence $p^2 \le |L| \le |M|=p^3$. 
If $L=M$, then $\A$ is a $\CI$-S-ring by 
Lemma~\ref{cigwr}. Let $L < M$. The center 
$Z(M)=\langle \phi_1 \rangle$, where 
$(\bar{a}_1,\bar{a}_2,\bar{a}_3)^{\phi_1}=
(\bar{a}_1,\bar{a}_2,\bar{a}_3\bar{a}_1)$. 
Notice that, since $L \approx_{\rm Cay} M$, 
$Z(M)$ is the unique subgroup of $L$ of order $p$ having the 
property that its elements fix both $\bar{a}_1$ and $\bar{a}_2$.  

Let $X \in \cS(\A)$ such that $X \subset U$.
By Remark~\ref{cisrings}, $\A_U$ is a $\CI$-S-ring, and hence 
Lemma~\ref{cinormal}(ii) can be applied to $\A$, $U$, $f_1$ 
and the generating set $\{a_1, a_2, a_3, a_6\}$. We find that 
$f_1^U=\id_U$, and so $X^f=X$.  

Let $X \in \cS(\A)$ such that $X \subset U_1 \setminus U$. 
By Remark~\ref{cisrings} and Eq.~\eqref{eq:U13}, 
the S-ring $\A_{U_1/Q}$ is a normal $\CI$-S-ring, and 
hence Lemma~\ref{cinormal}(ii) can be applied to $\A$, 
$U_1/Q$, $f_1$ and the generating set $\{\bar{a}_1, \bar{a}_2, 
\bar{a}_4\}$. We find that $f_1^{U_1/Q}=\id_{U_1/Q}$, and since 
$Q \le \rad(X)$, this implies that $X^{f_1}=X$.

Finally, let $X \in \cS(\A)$ such that $X \not\subset U_1 \cup U$. 
Then $X \subset Uy$ for some $y \in \langle a^4 \rangle$. 
Since $f_1^{U_1/Q}=\id_{U_1/Q}$, $(Qy)^{f_1}=Qy$, hence 
$Qy^{f_1}=Q^{f_1}y^{f_1}=(Qy)^{f_1}=Qy$, and 
$(Uy)^{f_1}=Uy^{f_1}=UQy^{f_1}=Uy$. 
Let $\delta : U \to Uy$ be the maping defined by 
$x \mapsto xy$ ($x \in U$), 
and let $\psi=\delta (f_1)^{Uy} \delta^{-1}$. 

We claim that $\psi \in \aut(\A_U)$. 
This is equivalent to say that, for every $Y \in \cS(\A_U)$ and 
$x \in U$, $(Yx)^\psi=Yx^\psi$. This follows along the line:
$$
(Yx)^\psi=(Yxy)^{f_1\delta^{-1}}=(Y^{f_1}(xy)^{f_1})^{\delta^{-1}}
=Y(xy)^{f_1}y^{-1}=Y x^\psi,
$$
where by the 2nd equality we used that $f_1 \in \iso_e(\A)$ and by the 
3rd one that $f_1^{U}=\id_U$. It is clear that 
$(U/Q)^\psi=U/Q$, hence $\psi^{U/Q} \in \aut(\A_U)^{U/Q}=
\aut_U(\A_U)^{U/Q}=L$.

Since $f_1^{U_1/Q}=\id_{U_1/Q}$, it follows that $\psi^{U/Q}$ fixes both $\bar{a}_1$ and $\bar{a}_2$, and hence $\psi^{U/Q} \in \langle \phi_1 \rangle$. Consequently, for every $x \in U$, $(P_1Qx)^\psi=
P_1Qx$. Combining this with the condition that 
$P_1Q \le \rad(X)$, we can write   
$$
X^{f_1}=(X{y^{-1}})^{\psi \delta}=
((P_1QXy^{-1})^\psi)^\delta=(P_1QXy^{-1})^\delta=P_1QX=X.
$$
This completes the proof of Case~2.
\end{proof}

\section*{Acknowledgement}
The authors would like to thank M.\ Muzychuk and  G.\ Somlai for  fruitful discussions on the subject.

\end{document}